\documentclass[12pt]{article}
\usepackage{amsmath,amssymb,amsthm}
\numberwithin{equation}{section}
\usepackage{units}
\usepackage{color}
\usepackage[T1]{fontenc}
\usepackage[utf8]{inputenc}
\usepackage{authblk}
\usepackage{bm}
\usepackage{graphicx}

\usepackage{datetime}

\usepackage[colorlinks=true,
linkcolor=webgreen,
filecolor=webbrown,
citecolor=webgreen]{hyperref}

\definecolor{webgreen}{rgb}{0,.5,0}
\definecolor{webbrown}{rgb}{.6,0,0}

\usepackage[toc,page]{appendix}

\newcommand{\Z}{{\mathbb Z}}

\newtheorem{thm}{Theorem}
\newtheorem{theorem}[thm]{Theorem}
\newtheorem{lemma}{Lemma}

\title{Fibonacci and Lucas Identities the Golden Way}

\author[]{Kunle Adegoke \\\href{mailto:adegoke00@gmail.com}{\tt adegoke00@gmail.com}}

\affil{Department of Physics and Engineering Physics, \mbox{Obafemi Awolowo University}, 220005 Ile-Ife, Nigeria}

\begin{document}
\date{}

\maketitle

\begin{abstract}
\noindent By expressing Fibonacci and Lucas numbers in terms of the powers of the golden ratio, $\alpha=(1+\sqrt 5)/2$ and its inverse, $\beta=-1/\alpha=(1-\sqrt 5)/2$, a multitude of Fibonacci and Lucas identities have been developed in the literature. In this paper, we follow the reverse course: we derive numerous Fibonacci and Lucas identities by making use of the well-known expressions for the powers of $\alpha$ and $\beta$ in terms of Fibonacci and Lucas numbers.

\end{abstract}
\section{Introduction}
The Fibonacci numbers, $F_n$, and the Lucas numbers, $L_n$, are defined, for \mbox{$n\in\Z$}, through the recurrence relations 
\begin{equation}\label{eq.s6z1bcx}
F_n=F_{n-1}+F_{n-2}, \mbox{($n\ge 2$)},\quad\mbox{$F_0=0$, $F_1=1$};
\end{equation}
and
\begin{equation}
L_n=L_{n-1}+L_{n-2}, \mbox{($n\ge 2$)},\quad\mbox{$L_0=2$, $L_1=1$};
\end{equation}
with
\begin{equation}
F_{-n}=(-1)^{n-1}F_n\,,\quad L_{-n}=(-1)^nL_n\,.
\end{equation}
Throughout this paper, we denote the golden ratio, $(1+\sqrt 5)/2$, by $\alpha$ and write $\beta=(1-\sqrt 5)/2=-1/\alpha$, so that $\alpha\beta=-1$ and $\alpha+\beta=1$. The following well-known algebraic properties of $\alpha$ and $\beta$ can be proved directly from Binet's formula for the $n$th Fibonacci number or by induction:
\begin{equation}\label{eq.y9hcktl}
\alpha^n=\alpha^{n-1}+\alpha^{n-2}\,,
\end{equation}
\begin{equation}\label{eq.wa1x3n6}
\beta^n=\beta^{n-1}+\beta^{n-2}\,,
\end{equation}
\begin{equation}\label{eq.hxtle1n}
\alpha ^n  = \alpha F_n  + F_{n - 1}\,, 
\end{equation}
\begin{equation}\label{eq.oos3hdl}
\alpha ^n \sqrt 5  = \alpha ^n (\alpha - \beta) = \alpha L_n  + L_{n - 1}\,, 
\end{equation}
\begin{equation}\label{eq.gt6xmho}
\beta ^n  = \beta F_n  + F_{n - 1}\,,
\end{equation}
\begin{equation}\label{eq.nag65l1}
\beta ^n\sqrt  5  = \beta ^n (\alpha - \beta) =  -\beta L_n  - L_{n - 1}\,,
\end{equation}
\begin{equation}\label{eq.eu0zd8d}
\beta ^n  =  - \alpha F_n  + F_{n + 1}\,,
\end{equation}
\begin{equation}\label{eq.eadpgp1}
\beta ^n \sqrt 5  =   \alpha L_n  - L_{n + 1}\,,
\end{equation}
\begin{equation}\label{eq.huo1cwh}
\alpha^{-n}=(-1)^{n-1}\alpha F_n+(-1)^nF_{n+1}
\end{equation}
and
\begin{equation}\label{eq.rl4090d}
\beta^{-n}=(-1)^{n}\alpha F_n+(-1)^nF_{n-1}\,.
\end{equation}
Hoggatt et.~al.~\cite{hoggatt71} derived, among other results, the identity
\[
F_{k+t}=\alpha^kF_t+\beta^tF_k\,,
\]
which, upon multiplication by $\alpha^s$, can be put in the form
\begin{equation}\label{eq.snj8qge}
\alpha^sF_{k+t}=\alpha^{s+k}F_t+(-1)^t\alpha^{s-t}F_k\,.
\end{equation}
Identity \eqref{eq.snj8qge} is unchanged under interchange of $k$ and $t$, $k$ and $s$ and interchange of $t$ and $-s$ and $s$ and $-t$; we therefore have three additional identities:
\begin{equation}\label{eq.nfsaicr}
\alpha^sF_{k+t}=\alpha^{s+t}F_k+(-1)^k\alpha^{s-k}F_t\,,
\end{equation}
\begin{equation}\label{eq.mnoup5u}
\alpha^kF_{s+t}=\alpha^{s+k}F_t+(-1)^t\alpha^{k-t}F_s
\end{equation}
and
\begin{equation}\label{eq.mnoup5u}
\alpha^{s-t}F_k=\alpha^{k-t}F_s+(-1)^s\alpha^{-t}F_{k-s}\,.
\end{equation}
As Koshy \cite[p.79]{koshy} noted, the two Binet formulas
\begin{equation}
F_n=\frac{\alpha^n-\beta^n}{\alpha-\beta},\quad L_n=\alpha^n+\beta^n\,,
\end{equation}
expressing $F_n$ and $L_n$ in terms of $\alpha^n$ and $\beta^n$, can be used in tandem to derive an array of identities. 

\medskip

Our aim in writing this paper is to derive numerous Fibonacci and Lucas identities by emphasizing identities \eqref{eq.hxtle1n} -- \eqref{eq.eadpgp1}, expressing $\alpha^n$ and $\beta^n$ in terms of $F_n$ and $L_n$. Our method relies on the fact that $\alpha$ and $\beta$ are irrational numbers. We will make frequent use of the fact that if $a$, $b$, $c$ and $d$ are rational numbers and $\gamma$ is an irrational number, then $a\gamma+b=c\gamma+d$ implies that $a=c$ and $b=d$; an observation that was used to advantage by Griffiths \cite{griffiths}.

\medskip

As a quick illustration of our method, take $x=\alpha F_p$ and $y=F_{p-1}$ in the binomial identity
\[
\sum_{j=0}^n{\binom njx^jy^{n-j}}=(x+y)^n\,,
\]
to obtain
\begin{equation}
\sum_{j = 0}^n {\binom nj\alpha ^j F_p^j F_{p - 1}^{n - j} }  = \alpha ^{np}\,,
\end{equation}
which, by multiplying both sides by $\alpha^q$, can be written
\begin{equation}\label{eq.n68cvcs}
\sum_{j = 0}^n {\binom nj\alpha ^{j+q} F_p^j F_{p - 1}^{n - j} }  = \alpha ^{np+q}\,,
\end{equation}
which, by identity \eqref{eq.hxtle1n}, evaluates to
\begin{equation}\label{eq.a1g1vk0}
\alpha \sum_{j = 0}^n {\binom njF_p^j F_{p - 1}^{n - j} F_{j + q} }  + \sum_{j = 0}^n {\binom njF_p^j F_{p - 1}^{n - j} F_{j + q - 1} }  = \alpha F_{np + q}  + F_{np + q - 1}\,.
\end{equation}
Comparing the coefficients of $\alpha$ in \eqref{eq.a1g1vk0}, we find
\begin{equation}\label{eq.c5vjr47}
\sum_{j = 0}^n {\binom njF_p^j F_{p - 1}^{n - j} F_{j + q} } = F_{np + q}\,;
\end{equation}
valid for non-negative integer $n$ and arbitrary integers $p$ and $q$.

\medskip

Identity \eqref{eq.c5vjr47} contains many known identities as special cases.

\medskip

If we write \eqref{eq.n68cvcs} as
\begin{equation}\label{eq.f807w1c}
\sum_{j = 0}^n {\binom nj\alpha ^{j+q} \sqrt 5F_p^j F_{p - 1}^{n - j} }  = \alpha ^{np+q}\sqrt 5
\end{equation}
and apply identity \eqref{eq.oos3hdl}, we obtain the Lucas version of \eqref{eq.c5vjr47}, namely,
\begin{equation}\label{eq.il7kcyt}
\sum_{j = 0}^n {\binom njF_p^j F_{p - 1}^{n - j} L_{j + q} } = L_{np + q}\,.
\end{equation}

\begin{lemma}
The following properties hold for $a$, $b$, $c$ and $d$ rational numbers:
\makeatletter
\renewcommand\tagform@[1]{\maketag@@@{\ignorespaces#1\unskip\@@italiccorr}}
\[\tag{P1}
a\alpha+b=c\alpha+d\iff a=c,\;b=d\,,
\]
\[\tag{P2}
\quad a\beta+b=c\beta+d\iff a=c,\;b=d\,,
\]
\[\tag{P3}
\quad\frac{1}{{c\alpha  + d}} = \left( {\frac{c}{{c^2  - d^2  - cd}}} \right)\alpha  - \left(\frac{{c + d}}{{c^2  - d^2  - cd}}\right)\,,
\]
\[\tag{P4}
\frac{1}{{c\beta  + d}} = \left( {\frac{c}{{c^2  - d^2  - cd}}} \right)\beta  - \left(\frac{{c + d}}{{c^2  - d^2  - cd}}\right)\,,
\]
\[\tag{P5}
\frac{{a\alpha  + b}}{{c\alpha  + d}} = \frac{{cb - da}}{{c^2  - d^2  - cd}}\alpha  + \frac{{ca - db - cb}}{{c^2  - cd - d^2 }}
\]
and
\[\tag{P6}
\frac{{a\beta  + b}}{{c\beta  + d}} = \frac{{cb - da}}{{c^2  - d^2  - cd}}\beta  + \frac{{ca - db - cb}}{{c^2  - cd - d^2 }}\,.
\]
\makeatother

\end{lemma}
Properties P3 to P6 follow from properties P1 and P2.

\medskip

The rest of this section is devoted to using our method to re-discover known results or to discover results that may be easily deduced from known ones. In establishing some of the identities we require the fundamental relations $F_{2n}=F_nL_n$, $L_n=F_{n-1}+F_{n+1}$ and $5F_n=L_{n-1}+L_{n+1}$. Presumably new results will be developed in section \ref{sec.auf9q1i}.
\subsubsection*{Fibonacci and Lucas addition formulas}
To derive the Fibonacci addition formula, use \eqref{eq.hxtle1n} to write the identity
\begin{equation}
\alpha^{p+q}=\alpha^p\alpha^q
\end{equation}
as
\[
\alpha F_{p + q}  + F_{p + q - 1}=(\alpha F_p  + F_{p - 1} )(\alpha F_q  + F_{q - 1} ) \,,
\]
from which, by multiplying out the right side, making use of \eqref{eq.hxtle1n} again, we find
\begin{equation}\label{eq.crrliq4}
\begin{split}
\alpha F_{p + q}  + F_{p + q - 1}&= \alpha(F_p F_q  + F_p F_{q - 1}  + F_{p - 1} F_q ) + F_p F_q  + F_{p - 1} F_{q - 1}\\
&=\alpha{(F_p F_{q + 1}  + F_{p - 1} F_q)} + F_p F_q  + F_{p - 1} F_{q - 1}\,.
\end{split}
\end{equation}
Equating coefficients of $\alpha$ (property P1) from both sides of \eqref{eq.crrliq4} establishes the well-known Fibonacci addition formula:
\begin{equation}\label{eq.uhyuo34}
F_{p + q}=F_p F_{q + 1}  + F_{p - 1} F_q\,.
\end{equation}
A similar calculation using the identity 
\begin{equation}
\alpha^p\beta^q=(-1)^q\alpha^{p-q}
\end{equation}
produces the subtraction formula
\begin{equation}\label{eq.zckqqm5}
(-1)^qF_{p-q}=F_pF_{q-1}-F_{p-1}F_q\,,
\end{equation}
which may, of course, be obtained from \eqref{eq.uhyuo34} by changing $q$ to $-q$.

\medskip

The Lucas counterpart of identity \eqref{eq.uhyuo34} is obtained by applying \eqref{eq.hxtle1n} and \eqref{eq.oos3hdl} to the identity
\begin{equation}
\alpha^{p+q}\sqrt 5=(\alpha^p\sqrt 5\,)\alpha^q\,,
\end{equation}
and proceeding as in the Fibonacci case, giving
\begin{equation}\label{eq.rr17kpl}
L_{p + q}=F_p L_{q + 1}  + F_{p - 1} L_q \,.
\end{equation}
Application of \eqref{eq.hxtle1n} to the right side and \eqref{eq.oos3hdl} to the left side of the identity
\begin{equation}
5\alpha ^{p + q}=(\alpha ^p \sqrt 5 \,)(\alpha ^q \sqrt 5 \,)
\end{equation}
produces
\begin{equation}
5F_{p + q}=L_p L_{q + 1}  + L_{p - 1} L_q\,.
\end{equation}
\subsubsection*{Fibonacci and Lucas multiplication formulas}
Subtracting identity \eqref{eq.snj8qge} from identity \eqref{eq.nfsaicr} gives
\begin{equation}\label{eq.rw1nse4}
F_t \left( {\alpha ^{s + k}  - ( - 1)^k \alpha ^{s - k} } \right) = F_k \left( {\alpha ^{s + t}  - ( - 1)^t \alpha ^{s - t} } \right)\,.
\end{equation}
Applying identity \eqref{eq.hxtle1n} to identity \eqref{eq.rw1nse4} and equating coefficients of $\alpha$, we obtain
\begin{equation}\label{eq.qvcfdjb}
F_t \left( {F_{s + k}  - ( - 1)^k F_{s - k} } \right) = F_k \left( {F_{s + t}  - ( - 1)^t F_{s - t} } \right)\,,
\end{equation}
which, upon setting $k=1$, gives
\begin{equation}
F_t L_s = F_{s + t}  - ( - 1)^t F_{s - t} \,.
\end{equation}
Writing identity \eqref{eq.rw1nse4} as
\begin{equation}\label{eq.waltf4g}
F_t \left( {\alpha ^{s + k}\sqrt 5  - ( - 1)^k \alpha ^{s - k}\sqrt 5 } \right) = F_k \left( {\alpha ^{s + t}\sqrt 5  - ( - 1)^t \alpha ^{s - t}\sqrt 5 } \right)\,,
\end{equation}
applying identity \eqref{eq.oos3hdl} and equating coefficients of $\alpha$ produces
\begin{equation}\label{eq.qrx0rof}
F_t \left( {L_{s + k}  - ( - 1)^k L_{s - k} } \right) = F_k \left( {L_{s + t}  - ( - 1)^t L_{s - t} } \right)\,,
\end{equation}
which, upon setting $k=1$, gives
\begin{equation}\label{eq.es6bhhx}
5F_t F_s = L_{s + t}  - ( - 1)^t L_{s - t} \,.
\end{equation}
Adding identity \eqref{eq.snj8qge} and identity \eqref{eq.nfsaicr}, making use of identity \eqref{eq.hxtle1n} and equating the coefficients of $\alpha$, we obtain
\begin{equation}
2F_{k + t} F_s  = F_t \left( {F_{s + k}  + ( - 1)^k F_{s - k} } \right) + F_k \left( {F_{s + t}  + ( - 1)^t F_{s - t} } \right)\,,
\end{equation}
which, at $k=t$ reduces to
\begin{equation}\label{eq.jfoym63}
L_t F_s = F_{s + t}  + ( - 1)^t F_{s - t} \,.
\end{equation}
Similarly, adding identity \eqref{eq.snj8qge} and identity \eqref{eq.nfsaicr}, multiplying through by $\sqrt 5$, making use of identity \eqref{eq.oos3hdl} and equating the coefficients of $\alpha$, we have
\begin{equation}
2F_{k + t} L_s  = F_t \left( {L_{s + k}  + ( - 1)^k L_{s - k} } \right) + F_k \left( {L_{s + t}  + ( - 1)^t L_{s - t} } \right)\,,
\end{equation}
which, at $k=t$ reduces to
\begin{equation}\label{eq.jfoym63}
L_t L_s = L_{s + t}  + ( - 1)^t L_{s - t} \,.
\end{equation}

\subsubsection*{Cassini's identity}
Since 
\begin{equation}\label{eq.ruemioz}
\alpha^n\beta^n=(\alpha\beta)^n=(-1)^n\,;
\end{equation}
applying identities \eqref{eq.hxtle1n} and \eqref{eq.eu0zd8d} to the left hand side of the above identity gives
\[
\begin{split}
\alpha ^n \beta ^n & = (\alpha F_n  + F_{n - 1} )( - \alpha F_n  + F_{n + 1} )\\
& =  - \alpha ^2 F_n^2  + \alpha (F_n F_{n + 1}  - F_n F_{n - 1} ) + F_{n - 1} F_{n + 1}\\
&= \alpha ( - F_n^2  + F_n^2 ) - F_n^2  + F_{n - 1} F_{n + 1}\,.
\end{split}
\]
Thus, according to \eqref{eq.ruemioz}, we have
\[
\alpha ( - F_n^2  + F_n^2 ) - F_n^2  + F_{n - 1} F_{n + 1}=(-1)^n\,.
\]
Comparing coefficients of $\alpha^0$ from both sides gives Cassini's identity:
\begin{equation}\label{eq.wk4gkn2}
F_{n-1}F_{n+1}=F_n^2+(-1)^n\,.
\end{equation}
To derive the Lucas version of \eqref{eq.wk4gkn2}, write
\begin{equation}
(\alpha^n\sqrt 5\,)(\beta^n\sqrt 5\,)=(-1)^n5\,;
\end{equation}
apply \eqref{eq.oos3hdl} and \eqref{eq.eadpgp1} to the left hand side, multiply out and equate coefficients, obtaining
\begin{equation}
L_{n-1}L_{n+1}-L^2_n=(-1)^{n-1}5\,.
\end{equation}
\subsubsection*{General Fibonacci and Lucas addition formulas and Catalan's identity}
From identity \eqref{eq.snj8qge}, we can derive an addition formula that includes identity \eqref{eq.uhyuo34} as a particular case.

\medskip

Using identity \eqref{eq.hxtle1n} to write the left hand side (lhs) and the right hand side (rhs) of identity \eqref{eq.snj8qge}, we have
\begin{equation}\label{eq.fg1qrsg}
\text{lhs of \eqref{eq.snj8qge}} = \alpha F_s F_{k + t}  + F_{s - 1} F_{k + t}
\end{equation}
and
\begin{equation}\label{eq.af1fj9i}
\begin{split}
\text{rhs of \eqref{eq.snj8qge}}&= \alpha F_{s + k} F_t  + F_{s + k - 1} F_t  + \alpha ( - 1)^t F_{s - t} F_k  + ( - 1)^t F_{s - t - 1} F_k\\
&= \alpha (F_{s + k} F_t  + ( - 1)^t F_{s - t} F_k ) + (F_{s + k - 1} F_t  + ( - 1)^t F_{s - t - 1} F_k )\,.
\end{split}
\end{equation}
Comparing the coefficients of $\alpha$ from \eqref{eq.fg1qrsg} and \eqref{eq.af1fj9i}, we find
\begin{equation}\label{eq.wo5odhd}
F_s F_{k + t}  = F_{s + k} F_t  + ( - 1)^t F_{s - t} F_k\,,
\end{equation}
of which identity \eqref{eq.uhyuo34} is a particular case.

\medskip

Setting $t=s-k$ in \eqref{eq.wo5odhd} produces Catalan's identity:
\begin{equation}
F_s^2  = F_{s + k} F_{s - k}  + ( - 1)^{s + k} F_k^2\,.
\end{equation}
Multiplying through identity \eqref{eq.snj8qge} by $\sqrt 5$ and performing similar calculations to above produces
\begin{equation}\label{eq.tzz9m9p}
L_s F_{k + t}  = L_{s + k} F_t  + ( - 1)^t L_{s - t} F_k\,,
\end{equation}
which at $t=s-k$ gives
\begin{equation}
F_{2s}=L_{s+k}F_{s-k}+(-1)^{s+k}F_{2k}\,.
\end{equation}
\subsubsection*{Sums of Fibonacci and Lucas numbers with subscripts in arithmetic progression}
Setting $x=\alpha^p$ in the geometric sum identity
\begin{equation}\label{eq.cmyl19p}
\sum_{j=0}^nx^j=\frac{1-x^{n+1}}{1-x}
\end{equation}
and multiplying through by $\alpha^q$ gives
\begin{equation}\label{eq.fqua2xh}
\sum_{j=0}^n\alpha^{pj+q}=\frac{\alpha^q-\alpha^{pn+p+q}}{1-\alpha^p}\,.
\end{equation}
Thus, we have
\begin{equation}
\sum_{j=0}^n\alpha^{pj+q}= {\frac{{\alpha (F_{pn + p + q}  - F_q ) + (F_{pn + p + q - 1}  - F_{q - 1} )}}{{\alpha F_p  + (F_{p - 1}  - 1)}}}\,,
\end{equation}
from which, with the use of identity \eqref{eq.hxtle1n} and property P5, we find
\begin{equation}
\sum_{j = 0}^n {F_{pj + q} }  = \frac{F_p (F_{pn + p + q - 1}  - F_{q - 1} )-(F_{p - 1}  - 1)(F_{pn + p + q} - F_q)}{{L_p  - 1 + ( - 1)^{p - 1} }}\,,
\end{equation}
valid for all integers $p$, $q$ and $n$. The derivation here is considerably simpler than in the direct use of Binet's formula as done, for example, in Koshy \cite[p. 86]{koshy} and by Freitag \cite{freitag}.

\medskip

Multiplying through identity \eqref{eq.fqua2xh} by $\sqrt 5$ gives
\begin{equation}
\sum_{j=0}^n\alpha^{pj+q}\sqrt 5=\frac{\alpha^q\sqrt 5-\alpha^{pn+p+q}\sqrt 5}{1-\alpha^p}\,,
\end{equation}
from which, by identities \eqref{eq.oos3hdl} and \eqref{eq.hxtle1n}, and properties P5 and P1, we find
\begin{equation}
\sum_{j = 0}^n {L_{pj + q} }  = \frac{F_p (L_{pn + p + q - 1}  - L_{q - 1} )-(F_{p - 1}  - 1)(L_{pn + p + q} - L_q)}{{L_p  - 1 + ( - 1)^{p - 1} }}\,.
\end{equation}
\subsubsection*{Generating functions of Fibonacci and Lucas numbers with indices in arithmetic progression}
Setting $x=y\alpha^p$ in the identity
\begin{equation}
\sum_{j = 0}^\infty  {x^j }  = \frac{1}{{1 - x}}
\end{equation}
and multiplying through by $\alpha^q$ gives
\begin{equation}
\sum_{j = 0}^\infty  {\alpha ^{pj + q} y^j }  = \frac{{\alpha ^q }}{{1 - \alpha ^p y}} = \frac{{F_q \alpha  + F_{q - 1} }}{{ - yF_p \alpha  + 1 - yF_{p - 1} }}\,.
\end{equation}
Application of identity \eqref{eq.hxtle1n} and properties P5 and P1 then produces
\begin{equation}\label{eq.pqhgv3t}
\sum_{j = 0}^\infty  {F_{pj + q} y^j }  = \frac{{F_q  + ( - 1)^q F_{p - q} y}}{{1 - L_p y + ( - 1)^p y^2 }}\,.
\end{equation}
To find the corresponding Lucas result, we write
\begin{equation}
\sum_{j = 0}^\infty  {\alpha ^{pj + q}\sqrt 5 y^j }  = \frac{{\alpha ^q\sqrt 5 }}{{1 - \alpha ^p y}}
\end{equation}
and use identity \eqref{eq.oos3hdl} and properties P5 and P1, obtaining
\begin{equation}\label{eq.sdhqsd6}
\sum_{j = 0}^\infty  {L_{pj + q} y^j }  = \frac{{L_q  - ( - 1)^q L_{p - q} y}}{{1 - L_p y + ( - 1)^p y^2 }}\,.
\end{equation}
Identity \eqref{eq.pqhgv3t}, but not \eqref{eq.sdhqsd6}, was reported in Koshy \cite[identity 18, p.230]{koshy}. The case $q=0$ in \eqref{eq.pqhgv3t} was first obtained by Hoggatt \cite{hoggatt71b} while the case $p=1$ in \eqref{eq.sdhqsd6} is also found in Koshy \cite[identity (19.2), p.231]{koshy}.
\section{Main results}\label{sec.auf9q1i}
\subsection{Recurrence relations}
\begin{theorem}
The following identities hold for integers $p$, $q$ and $r$:
\begin{equation}\label{eq.u6g3quu}
F_{p + q + r}  = F_{q - 1} F_r F_{p - 1} + (F_{q + 1} F_r  + F_{q - 1} F_{r - 1} ) F_p + F_q F_{r + 1} F_{p + 1}\,,  
\end{equation}                                                                                         
\begin{equation}\label{eq.hsiqhgv}                                                                                       
F_{p + q + r}  = F_{p - 1} F_r F_{q - 1} + (F_{p + 1} F_r  + F_{p - 1} F_{r - 1} ) F_q + F_p F_{r + 1} F_{q + 1}\,,   
\end{equation}                                                                                         
\begin{equation}                                                                                       
F_{p + q + r}  = F_{q - 1} F_p F_{r - 1} + (F_{q + 1} F_p  + F_{q - 1} F_{p - 1} ) F_r + F_q F_{p + 1} F_{r + 1} \,,  
\end{equation}                                                                                         
\begin{equation}\label{eq.j54wavw}                                                                                      
F_{p + q + r}  = F_{r - 1} F_q F_{p - 1} + (F_{r + 1} F_q  + F_{r - 1} F_{q - 1} ) F_p + F_r F_{q + 1} F_{p + 1}\,, 
\end{equation}

\medskip

\begin{equation}
L_{p + q + r}  = F_{q - 1} F_r L_{p - 1} + (F_{q + 1} F_r  + F_{q - 1} F_{r - 1} ) L_p + F_q F_{r + 1} L_{p + 1}\,,   
\end{equation}                                                                                         
\begin{equation}                                                                                       
L_{p + q + r}  = F_{p - 1} F_r L_{q - 1} + (F_{p + 1} F_r  + F_{p - 1} F_{r - 1} ) L_q + F_p F_{r + 1} L_{q + 1}\,,   
\end{equation}                                                                                         
\begin{equation}                                                                                       
L_{p + q + r}  = F_{q - 1} F_p L_{r - 1} + (F_{q + 1} F_p  + F_{q - 1} F_{p - 1} ) L_r + F_q F_{p + 1} L_{r + 1}\,,   
\end{equation}                                                                                         
\begin{equation}                                                                                       
L_{p + q + r}  = F_{r - 1} F_q L_{p - 1} + (F_{r + 1} F_q  + F_{r - 1} F_{q - 1} ) L_p + F_r F_{q + 1} L_{p + 1}\,,  
\end{equation}                                                                                           

\medskip

\begin{equation}
5F_{p + q + r}  = L_{q - 1} L_r F_{p - 1} + (L_{q + 1} L_r  + L_{q - 1} L_{r - 1} ) F_p + L_q L_{r + 1} F_{p + 1}\,,  
\end{equation}                                                                                         
\begin{equation}                                                                                       
5F_{p + q + r}  = L_{p - 1} L_r F_{q - 1} + (L_{p + 1} L_r  + L_{p - 1} L_{r - 1} ) F_q + L_p L_{r + 1} F_{q + 1}\,,   
\end{equation}                                                                                         
\begin{equation}                                                                                       
5F_{p + q + r}  = L_{q - 1} L_p F_{r - 1} + (L_{q + 1} L_p  + L_{q - 1} L_{p - 1} ) F_r + L_q L_{p + 1} F_{r + 1}\,,  
\end{equation}                                                                                         
\begin{equation}                                                                                       
5F_{p + q + r}  = L_{r - 1} L_q F_{p - 1} + (L_{r + 1} L_q  + L_{r - 1} L_{q - 1} ) F_p + L_r L_{q + 1} F_{p + 1}\,,  
\end{equation}                                                                                           

\medskip

\begin{equation}                                                                                      
5L_{p + q + r}  = L_{q - 1} L_r L_{p - 1} + (L_{q + 1} L_r  + L_{q - 1} L_{r - 1} ) L_p + L_q L_{r + 1} L_{p + 1}\,, 
\end{equation}                                                                                         
\begin{equation}                                                                                       
5L_{p + q + r}  = L_{p - 1} L_r L_{q - 1} + (L_{p + 1} L_r  + L_{p - 1} L_{r - 1} ) L_q + L_p L_{r + 1} L_{q + 1}\,,
\end{equation}                                                                                         
\begin{equation}                                                                                       
5L_{p + q + r}  = L_{q - 1} L_p L_{r - 1} + (L_{q + 1} L_p  + L_{q - 1} L_{p - 1} ) L_r + L_q L_{p + 1} L_{r + 1}\,,
\end{equation}                                                                                         
\begin{equation}\label{eq.y5dseg8}                                                                                      
5L_{p + q + r}  = L_{r - 1} L_q L_{p - 1} + (L_{r + 1} L_q  + L_{r - 1} L_{q - 1} ) L_p + L_r L_{q + 1} L_{p + 1}\,.
\end{equation}

\end{theorem}
\begin{proof}
Identities \eqref{eq.u6g3quu} -- \eqref{eq.y5dseg8} are derived by applying \eqref{eq.hxtle1n} and \eqref{eq.oos3hdl} to the following identities:
\[
\alpha ^{p + q + r}  = \alpha ^p \alpha ^q \alpha ^r\,,
\]
\[
\alpha ^{p + q + r} \sqrt 5  = (\alpha ^p \sqrt 5 )\alpha ^q \alpha ^r\,, 
\]
\[
5\alpha ^{p + q + r}  = (\alpha ^p \sqrt 5 )(\alpha ^q \sqrt 5 )\alpha ^r 
\]
and
\[
5\alpha ^{p + q + r} \sqrt 5  = (\alpha ^p \sqrt 5 )(\alpha ^q \sqrt 5 )(\alpha ^r \sqrt 5 )\,.
\]
Note that identities \eqref{eq.hsiqhgv} -- \eqref{eq.j54wavw} are obtained from identity \eqref{eq.u6g3quu} by interchanging two indices from $p$, $q$ and $r$.
\end{proof}
\begin{theorem}
The following identities hold for integers $p$, $q$, $r$, $s$ and $t$:
\begin{equation}\label{eq.qbbejfn}
F_{p + q - r} F_{t - s + r}  + F_{p + q - r - 1} F_{t - s + r - 1}  = F_{p - s} F_{t + q}  + F_{p - s - 1} F_{t + q - 1}\,,
\end{equation}
\begin{equation}\label{eq.aiu1z6u}
F_{p + q - r} L_{t - s + r}  + F_{p + q - r - 1} L_{t - s + r - 1}  = F_{p - s} L_{t + q}  + F_{p - s - 1} L_{t + q - 1}\,,
\end{equation}
\begin{equation}\label{eq.a12i18b}
L_{p + q - r} F_{t - s + r}  + L_{p + q - r - 1} F_{t - s + r - 1}  = L_{p - s} F_{t + q}  + L_{p - s - 1} F_{t + q - 1}
\end{equation}
and
\begin{equation}\label{eq.yo685jl}
L_{p + q - r} L_{t - s + r}  + L_{p + q - r - 1} L_{t - s + r - 1}  = L_{p - s} L_{t + q}  + L_{p - s - 1} L_{t + q - 1}\,.
\end{equation}
\end{theorem}
\begin{proof}
Identity \eqref{eq.qbbejfn} is proved by applying identity \eqref{eq.hxtle1n} to the identity 
\[
\alpha^{p+q-r}\alpha^{t-s+r}=\alpha^{p-s}\alpha^{t+q}\,,
\]
multiplying out the products and applying property P1.

\medskip

Identities \eqref{eq.aiu1z6u} and \eqref{eq.a12i18b} are derived by writing
\[
\alpha^{p+q-r}(\alpha^{t-s+r}\sqrt 5)=\alpha^{p-s}(\alpha^{t+q}\sqrt 5)
\]
and
\[
(\alpha^{p+q-r}\sqrt 5)\alpha^{t-s+r}=(\alpha^{p-s}\sqrt 5)\alpha^{t+q}\,,
\]
and applying identities \eqref{eq.hxtle1n} and \eqref{eq.oos3hdl} and property P1.

\medskip

Finally identity \eqref{eq.yo685jl} is derived from
\[
(\alpha^{p+q-r}\sqrt 5)(\alpha^{t-s+r}\sqrt 5)=(\alpha^{p-s}\sqrt 5)(\alpha^{t+q}\sqrt 5)\,.
\]
\end{proof}
\subsection{Summation identities}
\subsubsection{Binomial summation identities}
\begin{lemma}
The following identities hold for positive integer $n$ and arbitrary $x$ and $y$:
\begin{equation}\label{eq.g81atqv}
\sum_{j = 0}^n {\binom njy^j x^{n - j}}  = (x + y)^n\,,
\end{equation}
\begin{equation}\label{eq.xmac84j}
\sum_{j = 0}^n {( - 1)^j \binom nj(x + y)^j x^{n - j}}  = ( - 1)^n y^n\,,
\end{equation}
\begin{equation}\label{eq.y8wpfl4}
\sum_{j = 0}^n {( - 1)^j \binom njy^j (x + y)^{n - j} }  = x^n\,,
\end{equation}
\begin{equation}\label{eq.gt826zn}
\sum_{j = 0}^n {\binom njjy^{j - 1} x^{n - j}}  = n(x + y)^{n - 1}\,,
\end{equation}
\begin{equation}\label{eq.r7394o5}
\sum_{j = 0}^n {( - 1)^j \binom njj(x + y)^{j - 1} x^{n - j}}  = ( - 1)^n ny^{n - 1}
\end{equation}
and
\begin{equation}\label{eq.c0t71uc}
\sum_{j = 1}^n {( - 1)^{j - 1} \binom njy^{j - 1} j(x + y)^{n - j} }  = nx^{n - 1}\,.
\end{equation}

\end{lemma}
\begin{proof}
Setting $z=0$ in the binomial identity
\begin{equation}\label{eq.bx3ymai}
\sum_{j = 0}^n {\binom njy^j e^{jz} x^{n - j}}  = (x + ye^z )^n 
\end{equation}
gives identity \eqref{eq.g81atqv}. Identities \eqref{eq.xmac84j} and \eqref{eq.y8wpfl4} are obtained from identity \eqref{eq.g81atqv} by obvious transformations. Identity \eqref{eq.gt826zn} is obtained by differentiating identity \eqref{eq.bx3ymai} with respect to $z$ and then setting $z$ to zero. 
More generally,
\begin{equation}
\sum_{j = 0}^n {\binom njj^ry^jx^{n - j}}  = \left. {\frac{{d^r }}{{dz^r }}(x + ye^z )^n } \right|_{z = 0}\,.
\end{equation}

Identities \eqref{eq.r7394o5} and \eqref{eq.c0t71uc} are obtained from identity \eqref{eq.gt826zn} by transformations.

\end{proof}
\begin{theorem}
The following identities hold for integers $k$, $t$, $s$ and positive integer $n$:
\begin{equation}\label{eq.rrfqs7d}
\sum_{j = 0}^n {( - 1)^{tj} \binom njF_k^j F_t^{n - j} F_{(s + k)n - (t + k)j} }  = F_{t + k}^n F_{sn}\,,
\end{equation}
\begin{equation}\label{eq.k3iof5p}
\sum_{j = 0}^n {( - 1)^{tj} \binom njF_k^j F_t^{n - j} L_{(s + k)n - (t + k)j} }  = F_{t + k}^n L_{sn}\,,
\end{equation}
\begin{equation}\label{eq.hitv5kd}
\sum_{j = 0}^n {( - 1)^j \binom njF_{k + t}^j F_t^{n - j} F_{(s + k)n - kj} }  = ( - 1)^{n(t + 1)} F_k^n F_{n(s - t)}\,,
\end{equation}
\begin{equation}
\sum_{j = 0}^n {( - 1)^j \binom njF_{k + t}^j F_t^{n - j} L_{(s + k)n - kj} }  = ( - 1)^{n(t + 1)} F_k^n L_{n(s - t)}\,,
\end{equation}
\begin{equation}
\sum_{j = 0}^n {( - 1)^{(t + 1)j} \binom njF_k^j F_{k + t}^{n - j} F_{sn - tj} }  = F_t^n F_{n(s + k)}\,,
\end{equation}
\begin{equation}
\sum_{j = 0}^n {( - 1)^{(t + 1)j} \binom njF_k^j F_{k + t}^{n - j} L_{sn - tj} }  = F_t^n L_{n(s + k)}\,,
\end{equation}
\begin{equation}
(-1)^t\sum_{j = 1}^n {( - 1)^{tj} \binom njjF_k^{j - 1} F_t^{n - j} F_{(k + s)n + t - s - (k + t)j} }  = nF_{k + t}^{n - 1} F_{s(n - 1)}\,,
\end{equation}
\begin{equation}
(-1)^t\sum_{j = 1}^n {( - 1)^{tj} \binom njjF_k^{j - 1} F_t^{n - j} L_{(k + s)n + t - s - (k + t)j} }  = nF_{k + t}^{n - 1} L_{s(n - 1)}\,,
\end{equation}
\begin{equation}
\sum_{j = 1}^n {( - 1)^j \binom njjF_{k + t}^{j - 1} F_t^{n - j} F_{(k + s)n - s - kj} }  = ( - 1)^{n(t + 1) + t} nF_k^{n - 1} F_{(s - t)(n - 1)}\,,
\end{equation}
\begin{equation}
\sum_{j = 1}^n {( - 1)^j \binom njjF_{k + t}^{j - 1} F_t^{n - j} L_{(k + s)n - s - kj} }  = ( - 1)^{n(t + 1) + t} nF_k^{n - 1} L_{(s - t)(n - 1)}\,,
\end{equation}
\begin{equation}
( - 1)^{t + 1} \sum_{j = 1}^n {( - 1)^{(t + 1)j} \binom njjF_k^{j - 1} F_{k + t}^{n - j} F_{s(n - 1) + t - tj} }  = nF_t^{n - 1} F_{(s + k)(n - 1)}
\end{equation}
and
\begin{equation}\label{eq.igepvk1}
( - 1)^{t + 1} \sum_{j = 1}^n {( - 1)^{(t + 1)j} \binom njjF_k^{j - 1} F_{k + t}^{n - j} L_{s(n - 1) + t - tj} }  = nF_t^{n - 1} L_{(s + k)(n - 1)}\,.
\end{equation}

\end{theorem}
\begin{proof}
Choosing $x=\alpha^{s+k}F_t$ and $y=(-1)^t\alpha^{s-t}F_k$ in identity \eqref{eq.g81atqv} and taking note of identity \eqref{eq.snj8qge}, we have
\begin{equation}\label{eq.phr1465}
\sum_{j = 0}^n {( - 1)^{tj} \binom njF_k^j F_t^{n - j} \alpha ^{(s + k)n - (t + k)j} }  = F_{k + t}^n \alpha ^{ns}\,.
\end{equation}
Application of identity \eqref{eq.hxtle1n} and property P1 to \eqref{eq.phr1465} produces identity \eqref{eq.rrfqs7d}. To prove \eqref{eq.k3iof5p}, multiply through identity \eqref{eq.phr1465} by $\sqrt 5$ to obtain
\begin{equation}\label{eq.s1ga4rb}
\sum_{j = 0}^n {( - 1)^{tj} \binom njF_k^j F_t^{n - j} \{\alpha ^{(s + k)n - (t + k)j} }\sqrt 5\}  = F_{k + t}^n \{\alpha ^{ns}\sqrt 5\}\,.
\end{equation}
Use of identity \eqref{eq.oos3hdl} and property P1 in identity \eqref{eq.s1ga4rb} gives identity \eqref{eq.k3iof5p}. Identities \eqref{eq.hitv5kd} -- \eqref{eq.igepvk1} are derived in a similar fashion.
\end{proof}
\begin{lemma}\label{lem.lig8fdt}
The following identities hold true for integer $r$, non-negative integer $n$, and arbitrary $x$ and $y$:
\begin{equation}\label{eq.z4ctfaa}
\sum_{j = 0}^n {\frac{{2n + 1}}{{n + j + 1}}\binom {n+j+1}{2j+1}(xy)^{r - j} (x - y)^{2j + 1} }=x^{r + n + 1} y^{r - n}  - y^{r + n + 1} x^{r - n}\,,
\end{equation}
\begin{equation}\label{eq.tqhg3b9}
\sum_{j = 0}^n {\frac{{2n + 1}}{{n + j + 1}}\binom {n+j+1}{2j+1}(x(x-y))^{r - j} y^{2j + 1} }=x^{r + n + 1} (x - y)^{r - n}  - (x - y)^{r + n + 1} x^{r - n}
\end{equation}
and
\begin{equation}\label{eq.cju184j}
\sum_{j = 0}^n {\frac{{2n + 1}}{{n + j + 1}}\binom {n+j+1}{2j+1}(y(y-x))^{r - j} x^{2j + 1} }=y^{r + n + 1} (y-x)^{r - n}  - (y-x)^{r + n + 1} y^{r - n}\,.
\end{equation}

\end{lemma}
\begin{proof}
Jennings \cite[Lemma (i)]{jennings93} derived an identity equivalent to the following:
\begin{equation}\label{eq.i42bggo}
\sum_{j = 0}^n {\frac{{2n + 1}}{{n + j + 1}}\binom{n+j+1}{2j+1}\left( \frac{z^2-1}{z} \right)^{2j}}  = \frac{{z^2 z^{2n}  - z^{ - 2n} }}{{z^2  - 1}}\,.
\end{equation} 
Setting $z^2=x/y$ in the above identity and clearing fractions gives identity \eqref{eq.z4ctfaa}. Identity \eqref{eq.tqhg3b9} is obtained by replacing $y$ with $x-y$ in identity \eqref{eq.z4ctfaa}. Identity \eqref{eq.cju184j} is obtained by interchanging $x$ with $y$ in identity \eqref{eq.tqhg3b9}.

\end{proof}
\begin{theorem}
The following identities hold for non-negative integer $n$ and integers $s$, $k$, $r$ and $t$:
\begin{equation}\label{eq.ick3pfg}
\begin{split}
&( - 1)^t\sum_{j = 0}^n { \frac{{2n + 1}}{{n + j + 1}}\binom{n+j+1}{2j+1}(F_{k + t} F_t )^{r - j} F_k^{2j + 1} F_{r(2s + k) + s - t - (2t + k)j} }\\
&\qquad\qquad= F_{k + t}^{r + n + 1} F_t^{r - n} F_{s(r + n + 1) + (s + k)(r - n)}  - F_t^{r + n + 1} F_{k + t}^{r - n} F_{(s + k)(r + n + 1) + s(r - n)}\,,
\end{split}
\end{equation}

\bigskip

\begin{equation}
\begin{split}
&( - 1)^t\sum_{j = 0}^n {\frac{{2n + 1}}{{n + j + 1}}\binom{n+j+1}{2j+1}(F_{k + t} F_t )^{r - j} F_k^{2j + 1} L_{r(2s + k) + s - t - (2t + k)j} }\\
&\qquad\qquad= F_{k + t}^{r + n + 1} F_k^{r - n} L_{s(r + n + 1) + (s + k)(r - n)}  - F_t^{r + n + 1} F_{k + t}^{r - n} L_{(s + k)(r + n + 1) + s(r - n)}\,,
\end{split}
\end{equation}

\bigskip

\begin{equation}
\begin{split}
&\sum_{j = 0}^n {( - 1)^{tj} \frac{{2n + 1}}{{n + j + 1}}\binom{n+j+1}{2j+1}(F_{k + t} F_t )^{r - j} F_t^{2j + 1} F_{r(2s - t) + s + k + (2k + t)j} }\\
&\qquad\qquad= ( - 1)^{tn} F_{k + t}^{r + n + 1} F_k^{r - n} F_{s(2r + 1) - t(r - n)}  - ( - 1)^{tn + t} F_k^{r + n + 1} F_{k + t}^{r - n} F_{s(2r + 1) - t(r + n + 1)}\,,
\end{split}
\end{equation}

\bigskip

\begin{equation}
\begin{split}
&\sum_{j = 0}^n {( - 1)^{tj} \frac{{2n + 1}}{{n + j + 1}}\binom{n+j+1}{2j+1}(F_{k + t} F_t )^{r - j} F_t^{2j + 1} L_{r(2s - t) + s + k + (2k + t)j} }\\
&\qquad\qquad= ( - 1)^{tn} F_{k + t}^{r + n + 1} F_k^{r - n} L_{s(2r + 1) - t(r - n)}  - ( - 1)^{tn + t} F_k^{r + n + 1} F_{k + t}^{r - n} L_{s(2r + 1) - t(r + n + 1)}\,,
\end{split}
\end{equation}

\bigskip

\begin{equation}
\begin{split}
&\sum_{j = 0}^n {( - 1)^{(t-1)j} \frac{{2n + 1}}{{n + j + 1}}\binom{n+j+1}{2j+1}(F_t F_k )^{r - j} F_{k + t}^{2j + 1} F_{s(2r + 1) + (k - t)r - (k - t)j} }\\
&\qquad\qquad= ( - 1)^{(t - 1)n} F_t^{r + n + 1} F_k^{r - n} F_{s(2r + 1) + k(r + n + 1) + t(n - r)}\\
&\quad\qquad\qquad  - ( - 1)^{(t - 1)(n + 1)} F_k^{r + n + 1} F_t^{r - n} F_{s(2r + 1) - k(n - r) - t(r + n + 1)}
\end{split}
\end{equation}
and
\begin{equation}\label{eq.l7ndcj9}
\begin{split}
&\sum_{j = 0}^n {( - 1)^{(t-1)j} \frac{{2n + 1}}{{n + j + 1}}\binom{n+j+1}{2j+1}(F_t F_k )^{r - j} F_{k + t}^{2j + 1} L_{s(2r + 1) + (k - t)r - (k - t)j} }\\
&\qquad\qquad= ( - 1)^{(t - 1)n} F_t^{r + n + 1} F_k^{r - n} L_{s(2r + 1) + k(r + n + 1) + t(n - r)}\\
&\quad\qquad\qquad  - ( - 1)^{(t - 1)(n + 1)} F_k^{r + n + 1} F_t^{r - n} L_{s(2r + 1) - k(n - r) - t(r + n + 1)}\,.
\end{split}
\end{equation}

\end{theorem}
\begin{proof}
Each of identities \eqref{eq.ick3pfg} -- \eqref{eq.l7ndcj9} is proved by setting $x=\alpha^{s+k}F_t$ and $y=(-1)^t\alpha^{s-t}F_k$ in the identities of Lemma \ref{lem.lig8fdt} and taking note of identity \eqref{eq.snj8qge} while making use of identities \eqref{eq.hxtle1n} and \eqref{eq.oos3hdl} and property~P1.
\end{proof}
\subsubsection{Summation identities not involving binomial coefficients}
\begin{lemma}[{\cite[Lemma 1]{adegoke18}}]\label{lem.u4bqbkc}
Let $(X_t)$ and $(Y_t)$ be any two sequences such that $X_t$ and $Y_t$, $t\in\Z$, are connected by a three-term recurrence relation $hX_t=f_1X_{t-a}+f_2Y_{t-b}$, where $h$, $f_1$ and $f_2$ are arbitrary non-vanishing complex functions, not dependent on $t$, and $a$ and $b$ are integers. Then, the following identity holds for integer $n$:
\[
f_2 \sum_{j = 0}^n {f_1^{n - j} h^j Y_{t - na  - b  + a j} }  = h^{n + 1} X_t  - f_1^{n + 1} X_{t - (n + 1)a }\,. 
\]

\end{lemma}
\begin{theorem}
The following identities hold for integers $n$, $k$, $s$ and $t$:
\begin{equation}\label{eq.elbr7d1}
( - 1)^{nk + t - 1} F_k \sum_{j = 0}^n {( - 1)^{kj} F_{n(s - k) + s - t + 2kj} }  = F_t F_{(s + k)(n + 1)}  - F_{t + (n + 1)k} F_{s(n + 1)}\,,
\end{equation}
\begin{equation}\label{eq.zpc3vyq}
( - 1)^{nk + t - 1} F_k \sum_{j = 0}^n {( - 1)^{kj} L_{n(s - k) + s - t + 2kj} }  = F_t L_{(s + k)(n + 1)}  - F_{t + (n + 1)k} L_{s(n + 1)}\,,
\end{equation}
\begin{equation}\label{eq.t1gk197}
( - 1)^{nt + k - 1} F_t \sum_{j = 0}^n {( - 1)^{tj} F_{n(s - t) + s - k + 2tj} }  = F_k F_{(s + t)(n + 1)}  - F_{k + (n + 1)t} F_{s(n + 1)}\,,
\end{equation}
\begin{equation}
( - 1)^{nt + k - 1} F_t \sum_{j = 0}^n {( - 1)^{tj} L_{n(s - t) + s - k + 2tj} }  = F_k L_{(s + t)(n + 1)}  - F_{k + (n + 1)t} L_{s(n + 1)}\,,
\end{equation}
\begin{equation}
( - 1)^{ns + t - 1} F_s \sum_{j = 0}^n {( - 1)^{sj} F_{n(k - s) + k - t + 2sj} }  = F_t F_{(k + s)(n + 1)}  - F_{t + (n + 1)s} F_{k(n + 1)}
\end{equation}
and
\begin{equation}\label{eq.o9gxedr}
( - 1)^{ns + t - 1} F_s \sum_{j = 0}^n {( - 1)^{sj} L_{n(k - s) + k - t + 2sj} }  = F_t L_{(k + s)(n + 1)}  - F_{t + (n + 1)s} L_{k(n + 1)}\,.
\end{equation}

\end{theorem}
\begin{proof}
Write identity \eqref{eq.snj8qge} as $\alpha ^{s + k} F_t  = \alpha ^s F_{t + k}  + ( - 1)^{t - 1} \alpha ^{s - t} F_k $ and identify $h=\alpha^{s+k}$, $f_1=\alpha^s$, $f_2=F_k$, $a=-k$, $b=0$, $X_t=F_t$ and $Y_t=(-1)^{t-1}\alpha^{s-t}$ in Lemma \ref{lem.u4bqbkc}. Application of identity \eqref{eq.hxtle1n} to the resulting summation identity yields identity \eqref{eq.elbr7d1}. Identity \eqref{eq.zpc3vyq} is obtained by multiplying the $\alpha-$sum by $\sqrt 5$ and using identity \eqref{eq.oos3hdl}. Identities \eqref{eq.t1gk197} -- \eqref{eq.o9gxedr} are obtained from identities \eqref{eq.elbr7d1} and \eqref{eq.zpc3vyq} by interchanging $k$ and $t$; and $s$ and $k$, in turn, since identity \eqref{eq.snj8qge} remains unchanged under these operations.

\end{proof}
\begin{lemma}
The following identities hold for integers $r$ and $n$ and arbitrary $x$ and $y$:
\begin{equation}\label{eq.bdoy5wn}
(x - y)\sum_{j = 0}^n {y^{r - j} x^j }  = y^{r - n} x^{n + 1}  - y^{r + 1} \,,
\end{equation}
\begin{equation}\label{eq.jes4tt5}
x\sum_{j = 0}^n {y^{r - j} (x + y)^j }  = y^{r - n} (x + y)^{n + 1}  - y^{r + 1}
\end{equation}
and
\begin{equation}\label{eq.k66s9zf}
(x - y)\sum_{j = 0}^n {x^{r - j} y^j }  = x^{r + 1} - x^{r - n} y^{n + 1} \,.
\end{equation}

\end{lemma}
\begin{proof}
Identity \eqref{eq.bdoy5wn} is obtained by replacing $x$ with $x/y$ in identity \eqref{eq.cmyl19p}. Identity \eqref{eq.jes4tt5} is obtained by replacing $x$ with $x+y$ in identity \eqref{eq.bdoy5wn}. Finally, identity \eqref{eq.k66s9zf} is obtained by interchanging $x$ and $y$ in identity \eqref{eq.bdoy5wn}.
\end{proof}
\begin{theorem}
The following identities hold for integers $r$, $n$, $s$, $k$ and $t$:
\begin{equation}\label{eq.wdxi1rj}
( - 1)^t F_k \sum_{j = 0}^n {F_{k + t}^j F_t^{r - j} F_{r(s + k) + s - t - kj} }  = F_t^{r - n} F_{k + t}^{n + 1} F_{s(r + 1) + k(r - n)}  - F_t^{r + 1} F_{(s + k)(r + 1)}\,,
\end{equation}
\begin{equation}\label{eq.tsf4k42}
( - 1)^t F_k \sum_{j = 0}^n {F_{k + t}^j F_t^{r - j} L_{r(s + k) + s - t - kj} }  = F_t^{r - n} F_{k + t}^{n + 1} L_{s(r + 1) + k(r - n)}  - F_t^{r + 1} L_{(s + k)(r + 1)}\,,
\end{equation}
\begin{equation}\label{eq.cnhy6p9}
\begin{split}
&( - 1)^{rt} F_t \sum_{j = 0}^n {( - 1)^{tj} F_k^{r - j} F_{k + t}^j F_{r(s - t) + s + k + tj} }\\
&\qquad= F_k^{r - n} F_{k + t}^{n + 1} F_{r(s - t) + tn + s}  - ( - 1)^{t(r + 1)} F_k^{r + 1} F_{(s - t)(r + 1)}\,,
\end{split}
\end{equation}
\begin{equation}\label{eq.wchqq81}
\begin{split}
&( - 1)^{rt} F_t \sum_{j = 0}^n {( - 1)^{tj} F_k^{r - j} F_{k + t}^j L_{r(s - t) + s + k + tj} }\\
&\qquad= F_k^{r - n} F_{k + t}^{n + 1} L_{r(s - t) + tn + s}  - ( - 1)^{t(r + 1)} F_k^{r + 1} L_{(s - t)(r + 1)}\,,
\end{split}
\end{equation}
\begin{equation}\label{eq.gyojmdp}
( - 1)^t F_k \sum_{j = 0}^n {F_{k + t}^{r - j} F_t^j F_{s(r + 1) + s - t + kj} }  = F_{k + t}^{r + 1} F_{s(r + 1)}  - F_{k + t}^{r - n} F_t^{n + 1} F_{s(r + 1) + k(n + 1)}
\end{equation}
and
\begin{equation}\label{eq.y23kxxt}
( - 1)^t F_k \sum_{j = 0}^n {F_{k + t}^{r - j} F_t^j L_{s(r + 1) + s - t + kj} }  = F_{k + t}^{r + 1} L_{s(r + 1)}  - F_{k + t}^{r - n} F_t^{n + 1} L_{s(r + 1) + k(n + 1)}\,.
\end{equation}

\end{theorem}
\begin{proof}
Identities \eqref{eq.wdxi1rj} and \eqref{eq.tsf4k42} and identities \eqref{eq.gyojmdp} and \eqref{eq.y23kxxt} are obtained by setting $x=\alpha^{s}F_{k+t}$ and $y=\alpha^{s+k}F_t$ in identities \eqref{eq.bdoy5wn} and  \eqref{eq.k66s9zf} while taking note of identity \eqref{eq.snj8qge}. Identities \eqref{eq.cnhy6p9} and \eqref{eq.wchqq81} are derived by setting $x=\alpha^{s+k}F_t$ and $y=(-1)^t\alpha^{s-t}F_k$ in identity \eqref{eq.jes4tt5}.
\end{proof}
\begin{theorem}
The following identities hold for integers $p$, $q$ and $n$:
\begin{equation}
\begin{split}
\sum_{j = 0}^n {jF_{pj + q} } &= (n + 1)\frac{{F_p F_{p(n + 1) + q - 1}  - (F_{p - 1}  - 1)F_{p(n + 1) + q} }}{{L_p  - 1 + ( - 1)^{p - 1} }}\\
&\qquad  + \frac{{(F_{2p}  - 2F_p )(F_{p(n + 2) + q - 1}  - F_{p + q - 1} )}}{{(F_{2p - 1}  - 2F_{p - 1}  + 1)(F_{2p + 1}  - 2F_{p + 1}  + 1) - (F_{2p}  - 2F_p )^2 }}\\
&\quad\qquad - \frac{{(F_{2p - 1}  - 2F_{p - 1}  + 1)(F_{p(n + 2) + q}  - F_{p + q} )}}{{(F_{2p - 1}  - 2F_{p - 1}  + 1)(F_{2p + 1}  - 2F_{p + 1}  + 1) - (F_{2p}  - 2F_p )^2 }}
\end{split}
\end{equation}
and
\begin{equation}
\begin{split}
\sum_{j = 0}^n {jL_{pj + q} } &= (n + 1)\frac{{F_p L_{p(n + 1) + q - 1}  - (F_{p - 1}  - 1)L_{p(n + 1) + q} }}{{L_p  - 1 + ( - 1)^{p - 1} }}\\
&\qquad  + \frac{{(F_{2p}  - 2F_p )(L_{p(n + 2) + q - 1}  - L_{p + q - 1} )}}{{(F_{2p - 1}  - 2F_{p - 1}  + 1)(F_{2p + 1}  - 2F_{p + 1}  + 1) - (F_{2p}  - 2F_p )^2 }}\\
&\quad\qquad - \frac{{(F_{2p - 1}  - 2F_{p - 1}  + 1)(L_{p(n + 2) + q}  - L_{p + q} )}}{{(F_{2p - 1}  - 2F_{p - 1}  + 1)(F_{2p + 1}  - 2F_{p + 1}  + 1) - (F_{2p}  - 2F_p )^2 }}\,.
\end{split}
\end{equation}

\end{theorem}
\begin{proof}
Differentiating identity \eqref{eq.cmyl19p} with respect to $x$ and multiplying through by $x$ gives
\begin{equation}\label{eq.zgqrx49}
\sum_{j = 0}^n {jx^j }  = (n + 1)\frac{{x^{n + 1} }}{{x - 1}} - \frac{{x^{n + 2}  - x}}{{(x - 1)^2 }}\,.
\end{equation}
Setting $x=\alpha^p$ and multiplying through by $\alpha^q$ produces
\begin{equation}\label{eq.tmxc3h5}
\sum_{j = 0}^n {j\alpha ^{pj + q} }  = \frac{{(n + 1)\alpha ^{p(n + 1) + q} }}{{\alpha ^p  - 1}} + \frac{{\alpha ^{p(n + 2) + q}  - \alpha ^{p + q} }}{{2\alpha ^p  - \alpha ^{2p}  - 1}}\,,
\end{equation}
from which the results follow through the use of identities \eqref{eq.hxtle1n} and \eqref{eq.oos3hdl} and properties P1 and P5.

\end{proof}
\section{Concluding remarks}
A Fibonacci-like sequence $(G_k)_{k\in\Z}$ is one whose initial terms $G_0$ and $G_1$ are given integers, not both zero, and
\begin{equation}
G_k=G_{k-1}+G_{k-2},\quad G_{-k}=G_{-k+2}-G_{-k+1}\,.
\end{equation}
Identities \eqref{eq.y9hcktl} and \eqref{eq.wa1x3n6} show that the sequences $(A_k)_{k\in\Z}$ and $(B_k)_{k\in\Z}$, where $A_k=\alpha^k$ and $B_k=\beta^k$ are Fibonacci-like, with respective initial terms $A_0=\alpha^0=1$, $A_1=\alpha$ and $B_0=\beta^0=1$, $B_1=\beta=-1/\alpha$. The identity
\begin{equation}
F_{s-t}G_{k+m}=F_{m-t}G_{k+s}+(-1)^{s+t+1}F_{m-s}G_{k+t}\,,
\end{equation}
can therefore be written for the golden ration and its inverse as
\begin{equation}\label{eq.n44ackt}
F_{s-t}\alpha^{k+m}=F_{m-t}\alpha^{k+s}+(-1)^{s+t+1}F_{m-s}\alpha^{k+t}
\end{equation}
and
\begin{equation}\label{eq.m7m7v0c}
F_{s-t}\beta^{k+m}=F_{m-t}\beta^{k+s}+(-1)^{s+t+1}F_{m-s}\beta^{k+t}\,.
\end{equation}
Proceeding as in previous calculations, many identities can be derived using identities \eqref{eq.n44ackt} and \eqref{eq.m7m7v0c}. For example, setting $x=F_{m-t}\alpha^{k+s}$ and $y=(-1)^{s+t}F_{m-s}\alpha^{k+t}$ in the binomial identity
\[
\sum_{j=0}^n{(-1)^j\binom njy^jx^{n-j} }=(x-y)^n
\]
and multiplying through by $\alpha^p$ produces
\[
\sum_{j = 0}^n {( - 1)^{(s + t + 1)j} \binom njF_{m - s}^j F_{m - t}^{n - j} \alpha ^{(k + t)j + (k + s)(n - j) + p} }  = F_{s - t}^n \alpha ^{(k + m)n + p}\,, 
\]
from which we find
\[
\sum_{j = 0}^n {( - 1)^{(s + t + 1)j} \binom njF_{m - s}^j F_{m - t}^{n - j} F_{(s + k)n + p + (t - s)j} }  = F_{s - t}^n F_{(k + m)n + p}  
\]
and
\[
\sum_{j = 0}^n {( - 1)^{(s + t + 1)j} \binom njF_{m - s}^j F_{m - t}^{n - j} L_{(s + k)n + p + (t - s)j} }  = F_{s - t}^n L_{(k + m)n + p} \,. 
\]

\hrule

\noindent 2010 {\it Mathematics Subject Classification}:
Primary 11B39; Secondary 11B37.

\noindent \emph{Keywords: }
Horadam sequence, Fibonacci number, Lucas number, Fibonacci-like number, Generating function, Golden ratio.

\hrule



\end{document}